\documentclass[12pt]{amsart}
\usepackage{preamble}

\usepackage[headheight=15pt, headsep=15pt, footskip=27pt, bottom=2.4cm, left=2.4cm, right=2.4cm]{geometry}

\newcommand{\Alb}{\operatorname{Alb}}
\newcommand{\NS}{\operatorname{NS}}
\newcommand{\NE}{\overline{\operatorname{NE}}}
\newcommand{\edim}{\operatorname{edim}}
\newcommand{\MW}{\operatorname{MW}}
\newcommand{\Sym}{\operatorname{Sym}}

\newcommand{\Tot}{\operatorname{Tot}}
\usepackage{lineno}

\begin{document}

\title[Log canonical models of a fixed variety with varying boundaries]{Log canonical models of a fixed variety with varying boundaries}

\date{}

\begin{abstract}
Motivated by extending the Morrison-Kawamata cone conjecture beyond Calabi-Yau varieties and Severi-Maehara type finiteness results to targets not necessarily of general type, we fix a smooth projective variety and study the finiteness of its log canonical models as klt boundaries vary. Our main result establishes finiteness for every smooth projective minimal surface. For each $\kappa\in\{-\infty,0,1\}$, we construct a smooth projective non-minimal surface of Kodaira dimension $\kappa$ with infinitely many log canonical models, showing that the minimality assumption cannot be omitted in general. We also investigate possible extensions of this finiteness result to higher dimensions.
\end{abstract}

\author{Xingying Li}
\address[Xingying Li]{Department of Artificial Intelligence, Westlake University, Hangzhou 310030, China}
\email{xingyinglimath@gmail.com, lixingying@westlake.edu.cn}

\author{Zhan Li}
\address[Zhan Li]{Department of Mathematics, Southern University of Science and Technology, Shenzhen 518055, China}
\email{lizhan@sustech.edu.cn, lizhan.math@gmail.com}

\maketitle

\setcounter{tocdepth}{2}
\tableofcontents

\section{Introduction}\label{sec:introduction}

Throughout the paper, we work over the field of complex numbers.

Let $(X, \De)$ be a projective klt pair, where $K_X + \De$ is $\Qq$-Cartier. The finite-generation theorem
\cite[Corollary~1.1.2]{BCHM10} shows that the log canonical ring
\[
R(X, K_X + \De) = \bigoplus_{m\in \Nn} H^0\left(X, \Oo_X(\lf m(K_X+\De)\rf)\right)
\]
is finitely generated. If the Kodaira dimension of $K_X+\De$ is non-negative, then the natural map 
\[
X \dto \proj (R(X, K_X+\De))
\] is called the log canonical model of $(X, \De)$. Moreover, $Y \coloneqq \proj (R(X, K_X+\De))$ carries a divisor $\De_Y$ such that $(Y, \De_Y)$ is of log general type with klt singularities. 


Assuming standard conjectures in the minimal model program,
Shokurov's geography theory decomposes a fixed polytope of boundary
divisors on $X$ into finitely many polyhedral chambers on each of
which the log canonical model is constant; see \cite{SC11}. The Morrison-Kawamata cone conjecture suggests a global version of this phenomenon for Calabi-Yau varieties (i.e., $K_X$-trivial by our convention); see \cite{Mor93,Mor96,Kaw97,Tot10}. According to this conjecture, for a fixed Calabi-Yau variety $X$,
only finitely many isomorphism classes occur among the targets of the
log canonical models of $(X,\De)$ as $\De$ varies over all boundaries
for which $(X,\De)$ is klt, although the corresponding rational maps
may still be infinite in number.

On the other hand, the Severi-Maehara type theorem asserts that for a fixed variety $X$, the set
\[
\{Y \mid X \dto Y \text{~surjective rational map,~} Y \text{~general type}\}
\] is a finite set up to birational equivalence; see \cite{Mae83,Tsa97,BD97,HM06,BJR26}, etc. Note that $X\dto Y$ need not have connected generic fiber. Moreover, not only are the targets $Y$ finite up to birational equivalence, but so are the rational maps $X\dto Y$; see \cite{KO75}.

We investigate to what extent the Morrison-Kawamata cone conjecture extends to arbitrary varieties and the Severi-Maehara theorem extends to targets not necessarily of general type. More precisely, we seek to determine for which varieties the set of log canonical models is finite as the boundary varies.

For this purpose, we introduce the following notation. Let $X$ be a normal projective variety. Define
\[
\CC(X)\coloneqq
\{[Y]\mid X\dto Y \text{~is the log canonical model for some klt pair~}
(X,\De)\},
\]
where $[Y]$ denotes the isomorphism class of $Y$. Thus we record only
the isomorphism class of the target, not the corresponding rational map. We also classify these models according to their dimensions by defining
\[
\CC^d(X)\coloneqq \{[Y] \in \CC(X) \mid \dim Y=d\}.
\] Then one can ask the same question for each $\CC^d(X)$.

The natural first testing ground is the case of surfaces. Recall that a smooth projective surface $X$ is minimal if it contains no smooth rational curve with self-intersection $-1$. Our main result in this direction is the following.

\begin{theorem}\label{thm:minimal surfaces}
  Let $X$ be a smooth projective minimal surface. Then $\CC(X)$ is finite.
\end{theorem}

The minimality assumption can be dropped for the following classes of 
surfaces. 

\begin{theorem}\label{thm:intro-positive}
Let $X$ be a smooth projective surface. Then $\CC(X)$ is finite in each of the following cases:
\begin{enumerate}
  \item $X$ is Calabi-Yau;
\item $X$ is of general type;
\item $X$ is irregular.
\end{enumerate}
Moreover, $\CC^1(X)$ is finite for every smooth projective surface $X$.
\end{theorem}

Theorem~\ref{thm:intro-positive} is a combination of Theorem~\ref{thm:finite lc models for CY}, Theorem~\ref{thm:general-type-surface}, Theorem~\ref{thm:curve-targets} and Theorem~\ref{thm:irregular-surface} proved in Section~\ref{sec:positive}. Outside the cases covered above, however, the log canonical models can be infinite for non-minimal surfaces.

\begin{theorem}\label{thm:surface-counterexamples}
For every
\[
 \kappa\in\{-\infty,0,1\},
\]
there exists a smooth projective surface $X_\kappa$ with $\kappa(X_\kappa)=\kappa$ such that $\CC(X_\kappa)$ is infinite.
\end{theorem}

In hindsight, this is reasonable: blowing up general points can reduce the symmetry of $X$, making $\Aut(X)$ smaller, while increasing the complexity of its effective cone. It is therefore reasonable to impose a minimality condition on $X$. Although minimality is sufficient for the finiteness of log canonical models for surfaces, its direct higher-dimensional analogue is false.

\begin{theorem}\label{thm:intro-threefold}
There exists a smooth projective threefold $V$ with $K_V$ ample such that $\CC^3(V)$ is infinite.
\end{theorem}

The infinitely many log canonical models in Theorem~\ref{thm:intro-threefold}
are obtained by small contractions. It is interesting to compare this with
\cite[Theorem~4]{Tsa97}, which shows that for a smooth projective threefold
$V$ of general type,
\[
\{Y \mid V \to Y \text{ birational morphism with } Y \text{ smooth}\}
\]
is a finite set up to isomorphism. In fact, the proof of
\cite[Theorem~4]{Tsa97} extends to the case where the targets have
$\mathbb Q$-factorial terminal singularities. This suggests that, in order
to generalize Theorem~\ref{thm:minimal surfaces} to higher-dimensional
varieties, one may need to impose mildness conditions on the singularities
of the targets. Therefore, we consider the following variation.

For a normal projective variety $X$ of dimension $n$, set
\[
   \CC_{\mathrm{sm}}(X)
   \coloneqq
   \left\{
   [Y] \in \CC(X) \mid Y \text{ is smooth}
   \right\}.
\]

\begin{question}\label{que:smooth-minimal-model-finiteness}
Let $X$ be a smooth projective variety with $K_X$ nef. Is
$\CC_{\mathrm{sm}}(X)$ finite?
\end{question}

In view of Theorem~\ref{thm:surface-counterexamples}, the nefness of $K_X$
cannot be omitted in general. On the other hand, the smoothness assumption
on the target cannot be relaxed merely to $\mathbb Q$-factorial terminal
singularities; see Remark~\ref{rmk:Q-fact and terminal not enough}.

We conclude the introduction with an overview of the paper. Section~\ref{sec:preliminaries} fixes the terminology and notation. Section~\ref{sec:positive} proves the finiteness of log canonical models for Calabi-Yau surfaces, surfaces of general type and irregular surfaces. Moreover, we show the finiteness of curve targets. Theorem~\ref{thm:intro-positive} is a combination of these results. Moreover, Theorem~\ref{thm:minimal surfaces} is proved. Section~\ref{sec:surface-counterexamples} constructs non-minimal surfaces with infinitely many log canonical models. Theorem~\ref{thm:surface-counterexamples} is proved in this section. Section~\ref{sec:threefold} constructs a canonically polarized threefold with infinitely many log canonical models. Theorem~\ref{thm:intro-threefold} is proved in this section.

\subsection*{Acknowledgements}
The authors used ChatGPT to assist in the proof of Theorem~\ref{thm:minimal surfaces} and in the search for counterexamples. This work was partially supported by the Guangdong Basic and Applied Basic Research Foundation (No.~2024A1515012341) and the NSFC (No.~12471041).

\section{Preliminaries}\label{sec:preliminaries}

We collect the terminology used throughout the paper. A rational curve is an integral curve whose normalization is $\Pp^1$. For a normal projective variety $X$, its irregularity is
\[
q(X)\coloneqq h^1(X,\Oo_X),
\]
and its Kodaira dimension is $\kappa(X)\coloneqq\kappa(K_X)$. Let $X\to T$ be a projective morphism, where $X$ and $T$ are normal
quasi-projective varieties. We use ``$/T$'' to denote the relative
setting over $T$. For two $\Rr$-divisors $D$ and $B$ on $X/T$, we write
\[
D\sim_{\Rr}B/T
\qquad
(\text{resp. }D\equiv B/T)
\]
if they are $\Rr$-linearly equivalent over $T$ (resp. numerically equivalent over $T$). We write $\Eff(X/T) \subset N^1(X/T)$ for the cone generated by numerical classes of effective divisors over $T$ and $\NE(X/T)\subset N_1(X/T)$ for the closed cone generated by the numerical classes of effective curves contracted over $T$.

Let $(X,\De)$ be a pair, where $X$ is normal, $\De\geq0$ is an $\Rr$-divisor, and $K_X+\De$ is $\Rr$-Cartier. If $f\colon Y \to X$ is a proper birational morphism from a normal variety \(Y\) such that $D$ is a prime divisor on $Y$, then the log discrepancy of $D$ with respect to $(X, \De)$ is defined to be
\[a(D; X,\De) \coloneqq {\rm mult}_D(K_Y-f^*(K_X + \De)) + 1.\]
This definition is independent of the choice of $Y$. The pair $(X,\De)$ is klt if the log discrepancy of any divisor over $X$ is $> 0$.


Let
\(\phi\colon X \dashrightarrow Y/T\) be a birational contraction
(i.e., \(\phi\) does not extract divisors) of normal quasi-projective
varieties over \(T\), where \(Y\) is projective over \(T\). We write
\(\Delta_Y := \phi_*\Delta\) for the strict transform of \(\Delta\).
By \cite[Definition~3.6.7]{BCHM10}, \((Y/T,\Delta_Y)\) is a weak log canonical model of
\((X/T,\Delta)\) if \(K_Y+\Delta_Y\) is nef/\(T\) and
\[
a(D;Y,\Delta_Y) \geq a(D;X,\Delta)
\]
for any divisor \(D\) over \(X\). A weak log canonical model is called a good minimal model if \(K_Y+\Delta_Y\) is semiample/\(T\). Let $D$ be an $\Rr$-Cartier divisor on $X/T$. By \cite[Definition~3.6.5]{BCHM10}, a rational map
\[
g\colon X\dashrightarrow Z
\]
over $T$ is an ample model of $D$ over $T$ if $Z$ is normal and projective over $T$ and there exists an ample $\Rr$-Cartier divisor $H$ on $Z/T$ with the following property. If
\[
p\colon W\to X
\qquad\text{and}\qquad
q\colon W\to Z
\]
resolve $g$, then $q$ is a contraction morphism and
\[
p^*D\sim_{\Rr}q^*H+E/T
\]
for some effective divisor $E\geq0$ such that
\[
B\geq E
\]
for every $B\in|p^*D/T|_{\Rr}$, where $|p^*D/T|_{\mathbb{R}}$ denotes the $\Rr$-linear system of $p^*D$ over $T$. The ample model of $K_X+\De$ over $T$ is called the log canonical model of $(X/T,\De)$. For consistency of terminology, by our convention, a log canonical model need not be birational to the original variety (cf. \cite{BCHM10}). Moreover, if $\phi: X \dto Y/T$ is a weak log canonical model of a klt pair $(X/T,\De)$ with $X$ a $\Qq$-factorial variety, then $\phi$ can be realized as a log canonical model of a possibly a new klt pair $(X/T, \De')$. Indeed, let $A$ be an ample effective divisor on $Y/T$ and let $A_X$ be the strict transform of $A$ on $X$. Suppose that $p\colon W \to X$ and $q\colon W\to Y$ are projective birational morphisms such that $\phi=q\circ p^{-1}$. Let $\Exc(\phi)$ be the reduced sum of the exceptional prime divisors on $X$ contracted by $\phi$. By the negativity lemma (\cite[Lemma 3.39]{KM98}), there exists $m \gg 1$ such that
\[
p^*(A_X+m\Exc(\phi))=q^*A+E,
\] where $E \geq 0$ is a $q$-exceptional divisor. As $\phi$ is a weak log canonical model of $(X/T, \De)$, we have
\[
p^*(K_X+\De)=q^*(K_Y+\De_Y)+E',
\] where $E'\geq 0$ is $q$-exceptional. Therefore, for any $\ep>0$, we have
\[
p^*(K_X+\De+\ep(A_X+m\Exc(\phi)))=q^*(K_Y+\De_Y+\ep A)+(\ep E+E').
\] As $K_Y+\De_Y+\ep A$ is ample over $T$, as long as $(X/T, \De+\ep(A_X+m\Exc(\phi)))$ is klt, $\phi$ is the log canonical model of $(X/T, \De+\ep(A_X+m\Exc(\phi)))$.

For a projective variety $X$, if $\phi\colon X \dto Y$ is a good minimal model of $(X, \De)$ such that $h\colon Y \to Z$ is the contraction induced by $K_Y+\De_Y$, then $h\circ \phi \colon X \dto Z$ is the log canonical model of $(X,\De)$. Moreover, if $K_X+\De$ is $\Qq$-Cartier and
$\kappa(K_X+\De)\geq0$, then the log canonical model of $(X,\De)$ is
exactly the natural map
\[
X\dashrightarrow\proj R(X,K_X+\De).
\]

For pairs with singularities worse than klt, the existence of log
canonical models is not known in full generality. Moreover, the
Morrison-Kawamata cone conjecture can fail beyond the klt setting;
see \cite{Tot10}. Therefore, we restrict our discussion to klt pairs.
For simplicity, most of our results concern smooth underlying varieties.

Finally, for a vector bundle $E$ on a scheme $B$, we write
\[
\Tot_B(E)\coloneqq\spec_B\Sym(E^\vee)
\]
for its total space. We use Grothendieck's quotient convention
\[
\Pp_B(E)\coloneqq\proj_B\Sym(E)
\]
for the associated projective bundle. 

\section{Surfaces with finitely many log canonical models}\label{sec:positive}

\subsection{Calabi-Yau surfaces}

A projective surface $X$ with klt singularities such that
$K_X\sim_{\Qq}0$ is called a Calabi-Yau surface. The
Morrison-Kawamata cone conjecture is known for Calabi-Yau surfaces;
see \cite{Kaw97,Tot10}. The following finiteness result is the starting
point of this project.

\begin{theorem}\label{thm:finite lc models for CY}
  Let $X$ be a Calabi-Yau surface. Then $\CC(X)$ is a finite set.
\end{theorem}
\begin{proof}
  The Morrison-Kawamata cone conjecture for Calabi-Yau surfaces was
established in \cite[Theorem~4.1]{Tot10}. Together with
\cite[Theorem~1.4]{GLSW24}, this implies that there exists a rational
polyhedral cone $\Pi\subset\Eff(X)$ such that
  \begin{equation}\label{eq:fill Eff}
     \bigcup_{g\in \Aut(X)} g\cdot \Pi = \Eff(X),
  \end{equation} where $\Aut(X)$ is the automorphism group of $X$ which acts naturally on numerical classes of divisors. By the geography of minimal models (see \cite[Theorem~3.4]{SC11} or \cite[Theorem~2.7]{LZ25}),
$\Pi$ admits a finite decomposition into relatively open rational
polyhedral cones:
\[
\Pi=\bigcup_{1\leq j\leq m}P_j.
\]
This decomposition has the following property. Suppose that $(X,B)$
and $(X,D)$ are klt pairs with $[B],[D]\in P_j$. If
$\phi\colon X\dto Y$ is a weak log canonical model of $(X,B)$, then it
is also a weak log canonical model of $(X,D)$. Fix a map $\phi_j\colon X\dto Y_j$ for each $j$, and let $P_{Y_j}$
denote the image of $P_j$ under the pushforward map induced by
$\phi_j$. Since weak log canonical models of surfaces are good minimal
models,  and since $P_{Y_j}$ is relatively open, there is
a unique contraction
\[
Y_j\to Z_j
\]
induced by any semiample divisor $D$ whose class lies in $P_{Y_j}$.

 Now let $(X,\De)$ be a klt pair. By \eqref{eq:fill Eff}, there exists
$g\in\Aut(X)$ such that $g\cdot[\De]\in\Pi$. Choose $j$ such that
$g\cdot[\De]\in P_j$. By the preceding discussion,
  \[X \xrightarrow{g} X \overset{\phi_j}{\dashrightarrow}Y_j \to Z_{j}\] is the log canonical model of $(X, \De)$. Therefore, 
  \[\CC(X)=\{[Z_{j}]\mid 1 \leq j\leq m\},\] which is finite.
\end{proof}

\begin{remark}
The same argument works for any Calabi-Yau variety $X$ satisfying
the Morrison-Kawamata cone conjecture, provided that every klt pair
$(X,\De)$ admits a good minimal model. More generally, \cite{CLLZ25}
axiomatizes the Morrison-Kawamata cone conjecture and introduces
Morrison-Kawamata dream spaces (MKD spaces). The same argument applies
to MKD spaces, which form a broader class than varieties of
Calabi-Yau type. It remains an open question to classify MKD surfaces.
\end{remark}

\subsection{Surfaces of general type}

We show that smooth surfaces of general type admit only finitely many
log canonical models. For this purpose, we first establish several
auxiliary results.

\begin{proposition}\label{prop:finite-negative-curves}
Let $X$ be a smooth projective surface. Suppose that $X$ contains only finitely many smooth rational curves with negative self-intersection. Then $\CC^2(X)$ is a finite set. 
\end{proposition}

\begin{proof}
Let $f\colon X\dashrightarrow Z$ be the log canonical model of a klt
pair $(X,\Delta)$ such that $K_X+\Delta$ is big. Since the log minimal
model program for surfaces has no flips, $f$ factors as
\[
X\xrightarrow{g}Y\xrightarrow{h}Z,
\]
where $g$ is a good minimal model of $(X,\De)$ and $h$ is the morphism induced by
the semiample divisor $K_Y+g_*\Delta$. In particular, $f=h\circ g$ is
a morphism. The induced pair $(Z,f_*\Delta)$ is klt. Hence $Z$ has
rational singularities by \cite[Theorem~5.22]{KM98}, and consequently
$R^1f_*\Oo_X=0$. It follows from \cite[Lemma~3.8, p.~30]{Bad01} that
every irreducible component of $\Exc(f)$ is a smooth rational curve.

The intersection matrix of the exceptional curves is negative
definite. In particular, every irreducible component of $\Exc(f)$ has
negative self-intersection and therefore belongs to a fixed finite
collection of curves by assumption. By the rigidity lemma
\cite[Lemma~1.15]{Deb01}, a contraction is determined, up to
isomorphism of its target, by the curves that it contracts. Therefore,
only finitely many targets can occur.
\end{proof}

\begin{remark}
The proof above establishes the finiteness of the contractions, not
merely the finiteness of their targets.
\end{remark}

The geometric Lang conjecture predicts that, on a smooth projective variety of general type, the union of positive-dimensional subvarieties which are not of general type is a proper closed subset. For a surface, this would imply finiteness of all rational and elliptic curves, including singular ones. The following theorem is sufficient for our purposes.

\begin{theorem}[{\cite[Corollary~1]{LM95}}]\label{thm:Lu-Miyaoka}
A smooth projective surface of general type contains only finitely many smooth rational curves and smooth elliptic curves.
\end{theorem}

\begin{remark}
Theorem~\ref{thm:Lu-Miyaoka} asserts only the finiteness of smooth
rational and elliptic curves. The finiteness of all rational and
elliptic curves on an arbitrary surface of general type remains open.
\end{remark}

The preceding two results imply the finiteness of log canonical models
for smooth surfaces of general type.

\begin{theorem}\label{thm:general-type-surface}
Let $X$ be a smooth projective surface of general type. Then $\CC(X)$
is finite.
\end{theorem}

\begin{proof}
By Theorem~\ref{thm:Lu-Miyaoka}, $X$ contains only finitely many smooth
rational curves. Since $X$ is of general type, we have
$\CC(X)=\CC^2(X)$. The claim follows from
Proposition~\ref{prop:finite-negative-curves}.
\end{proof}

\subsection{Log canonical models which are curves}

We need the following finiteness theorem for abelian subvarieties.

\begin{theorem}[{\cite[Theorem]{LOZ96}}]\label{thm:loz}
Let $K$ be a field and let $A$ be an abelian variety over $K$. The natural action of $\Aut_K(A)$ on the set of abelian subvarieties of $A$ defined over $K$ has only finitely many orbits. Consequently, up to $K$-isomorphism, only finitely many abelian varieties can be embedded into $A$ as abelian subvarieties.
\end{theorem}

We show that the 1-dimensional log canonical models are always finite.

\begin{theorem}\label{thm:curve-targets}
Let $X$ be a smooth projective surface. Then $\CC^1(X)$ is finite.
\end{theorem}

\begin{proof}
Let
\[
 f\colon X\dto C
\]
be the log canonical model of a klt pair $(X,\Delta)$ with $\kappa(K_X+\Delta)=1$. By the log surface minimal model program, $f$ is a contraction morphism.

If $C\simeq\Pp^1$, there is only one isomorphism class.

Suppose that \(g(C)\geq 2\). We show that there is at most one such
contraction up to isomorphism of the target. Let
\[
f\colon X\to C
\qquad\text{and}\qquad
f'\colon X\to B
\]
be two log canonical model contractions with
\[
g(C),~g(B)\geq 2.
\]

If $f'$ contracts a fiber $f^{-1}(c)$ of $f$, then, by the rigidity
lemma \cite[Lemma~1.15]{Deb01}, there exists a Zariski open
neighborhood $U\subset C$ of $c$ and a factorization
\[f'|_{f^{-1}(U)}\colon f^{-1}(U) \to U \xrightarrow{h} B.\] As $C$ and $B$ are smooth projective curves, $h$ extends to a morphism $\ti h \colon C \to B$ such that $f'= \ti h \circ f$. Because both $f$ and $f'$ have connected fibers, the finite morphism $\ti h\colon C\to B$ has degree one and thus $\ti h$ is an isomorphism. Consequently, if \(f\) and \(f'\) are not isomorphic as contractions,
then \(f'\) contracts no fiber of \(f\). Hence the restriction of \(f'\)
to every fiber of \(f\) is nonconstant, and therefore surjective onto
\(B\). It follows that
\[
\theta\coloneqq(f,f')\colon X\to C\times B
\]
is dominant. Since \(X\) and \(C\times B\) are surfaces, \(\theta\)
is generically finite. By the ramification formula, there exists an effective divisor \(R\) on
\(X\) such that
\[
K_X\sim \theta^*K_{C\times B}+R.
\]
Since \(g(C),g(B)\geq 2\), the canonical
divisor \(K_{C\times B}\) is ample. Thus \(K_X\) is big.
This gives $\kappa(K_X+\Delta)=2$, contrary to the initial assumption $\kappa(K_X+\Delta)=1$. Therefore, there is at most one 1-dimensional log canonical model of $X$ of genus at least two, up to isomorphism. Alternatively, the finiteness also follows from the Severi-Maehara theorem; see \cite[Main Theorem]{Mae83}.

Finally, suppose that $C$ is an elliptic curve. Choose a point
$x_0\in X$ and take $f(x_0)$ as the origin of $C$. By the universal
property of the Albanese morphism, $f$ factors as
\[
X\xrightarrow{a_X}A\coloneqq\Alb(X)\xrightarrow{q}C,
\]
where $q$ is a surjective homomorphism of abelian varieties.

We claim that the kernel $\Ker(q)$ of $q$ is connected. Let $\Ker(q)^0$ be the identity component of $\Ker(q)$. We have 
\[
f\colon X \xrightarrow{a_X} A \to A/\Ker(q)^0 \to A/\Ker(q) \simeq C.
\] If $\Ker(q) \neq \Ker(q)^0$, then for any point $c_0$ of $C$, the preimage $f^{-1}(c_0)$ is disconnected. This is a contradiction.

If $\tau\colon A\to A$ is an origin-preserving automorphism, then, for every abelian subvariety
$K\subset A$, we have $A/K\simeq A/\tau(K)$. Therefore, by
Theorem~\ref{thm:loz}, there are only finitely many possibilities for
$C\simeq A/\Ker(q)$.
\end{proof}

\begin{remark}
There can be infinitely many contraction morphisms with the same elliptic
curve as target. For each primitive pair \((m,n)\in\mathbb Z^2\), the morphism
\[
f_{m,n}\colon E\times E\longrightarrow E,
\qquad (x,y)\longmapsto mx+ny,
\]
has connected fibers and can be realized as the log canonical model of a klt
pair.

Moreover, the connectedness of the fibers is essential for the finiteness of the
targets. Indeed, there exist infinitely many pairwise non-isomorphic elliptic
curves \(E_i\) admitting isogenies \(\varphi_i\colon E\to E_i\). Consequently,
the morphisms
\[
E\times\Pp^1\longrightarrow E_i,
\qquad
(x,t)\longmapsto\varphi_i(x),
\]
have disconnected fibers and pairwise non-isomorphic targets. The argument for elliptic-curve targets extends directly to
abelian-variety targets. However, we do not know whether finiteness
continues to hold when the targets are Calabi-Yau varieties.
\end{remark}

\subsection{Irregular surfaces}

Recall that a normal projective surface $X$ is irregular if $q(X)>0$.

\begin{theorem}\label{thm:irregular-surface}
Let $X$ be a smooth projective irregular surface. Then $\CC(X)$ is finite.
\end{theorem}

\begin{proof}
By Theorem~\ref{thm:curve-targets}, it suffices to prove that
$\CC^2(X)$ is finite, since the only zero-dimensional target is a
point.

Let
\[
a_X\colon X\longrightarrow \Alb(X)
\]
be the Albanese morphism, and write its Stein factorization as
\[
X\xrightarrow{h}S\xrightarrow{\nu}a_X(X).
\]
Since $q(X)>0$, we have $\dim S\in\{1,2\}$. Every rational curve on $X$ is contracted by $a_X$, because every
morphism from $\Pp^1$ to an abelian variety is constant. Since $\nu$
is finite, every rational curve on $X$ is therefore contracted by
$h$.

Suppose first that $\dim S=2$. Then $h$ is birational, and hence it
has only finitely many exceptional curves. Thus $X$ contains only
finitely many rational curves. Suppose now that $\dim S=1$. Every rational curve on $X$ is contained
in a fiber of $h$. All but finitely many fibers of \(h\) are smooth and irreducible. A negative rational curve must therefore be a component of one of the finitely many reducible fibers. Consequently, there are only finitely many rational curves with negative self-intersection. In either case, Proposition~\ref{prop:finite-negative-curves} shows
that $\CC^2(X)$ is finite.
\end{proof}

\subsection{Minimal surfaces}\label{subsec:minimal surfaces}

We prove Theorem~\ref{thm:minimal surfaces}, which states that for a
smooth projective minimal surface $X$, the set $\CC(X)$ is finite.

\begin{proof}[Proof of Theorem~\ref{thm:minimal surfaces}]
By the classification of minimal surfaces, a smooth projective minimal surface $X$ belongs to one of the following classes; see \cite[Chapter VI]{BHPV04}.

(I) A minimal rational surface is either $\Pp^2$ or a Hirzebruch surface
$\mathbb F_n$ with $n=0$ or $n\geq2$. All these surfaces are toric and
hence are Mori dream surfaces. It follows that $\CC(X)$ is finite; see \cite{HK00} for more details.

(II) A minimal ruled surface over a curve of positive genus is irregular. Thus, $\CC(X)$ is finite by Theorem~\ref{thm:irregular-surface}.

(III) A minimal surface of Kodaira dimension $0$ is a Calabi-Yau surface. Thus, $\CC(X)$ is finite by Theorem~\ref{thm:finite lc models for CY}.

(IV) For a minimal surface $X$ of general type, $\CC(X)$ is finite by Theorem~\ref{thm:general-type-surface}.

(V) The remaining class is a minimal properly elliptic surface,
which has Kodaira dimension $1$. A minimal properly elliptic surface over a curve of positive genus is irregular. Thus, $\CC(X)$ is finite by Theorem~\ref{thm:irregular-surface}. Hence the only remaining case is that of a minimal properly elliptic surface over $\Pp^1$. To show the finiteness of log canonical models of such surfaces, we argue in several steps.

\medskip
\noindent
\textit{Step 1: a uniform bound for horizontal smooth rational curves.}

Let
\[
    \pi:X\longrightarrow \Pp^1
\]
be a smooth projective relatively minimal genus-one surface with
$\kappa(X)=1$. If $q(X)>0$, then the assertion follows from Theorem~\ref{thm:irregular-surface}.
We therefore assume throughout the proof that
\[
    q(X)=0.
\]
Let $F$ denote the numerical class of a general fiber.  By the canonical bundle formula and $\kappa(X)=1$, we have
\begin{equation}\label{eq:elliptic-K-aF}
    K_X\equiv aF,
    \qquad a\in\Qq_{>0},
    \qquad K_X^2=F^2=0.
\end{equation}

Let $C\subset X$ be a horizontal smooth rational curve, and put
\[
    d_C:=C\cdot F>0.
\]
By adjunction and \eqref{eq:elliptic-K-aF},
\begin{equation}\label{eq:elliptic-rational-square-general}
    C^2=-2-a d_C.
\end{equation}
We apply \cite[Theorem~1.3(i)]{Miy08} with $\alpha=1$ and $g(C)=0$.  The
inequality in \cite[Theorem~1.3(i)]{Miy08} reads
\[
 \frac12\bigl(C^2+3K_X\cdot C+6\bigr)
 -2\bigl(K_X\cdot C+3\bigr)
 +3c_2(X)-K_X^2\geq0.
\]
Plugging in \eqref{eq:elliptic-K-aF} and \eqref{eq:elliptic-rational-square-general}, we obtain
\[
    K_X\cdot C\leq 3c_2(X)-4.
\]
Hence
\begin{equation}\label{eq:uniform-degree-bound-general-genus-one}
    1\leq d_C\leq
    M_X:=\max\left\{1,
    \left\lfloor\frac{3c_2(X)-4}{a}\right\rfloor\right\}.
\end{equation}
Thus both $d_C$ and $C^2$ range over finite sets.

There are only finitely many vertical smooth rational curves on $X$.  Indeed,
a smooth fiber has genus one, so every vertical smooth rational curve is an
irreducible component of one of the finitely many non-smooth fibers.

\medskip
\noindent
\textit{Step 2: a finitely generated group of Jacobian translations.}

Set
\[
    K\coloneqq \Cc(\Pp^1),\qquad T\coloneqq X_\eta,
    \qquad J\coloneqq \operatorname{Pic}^0_{T/K}.
\] Here \(K\) is the function field of \(\mathbb P^1\), \(T\) is the generic fiber, a genus-one curve over \(K\), and \(J\) is its Jacobian elliptic curve. The curve $T$ is a 
principal homogeneous space under $J$; see \cite[Chapter~X, Exercise~10.3(a) and Theorem~3.8]{Sil09}. Thus every
$Q\in J(K)$ acts on $T$ by a translation
\[
    t_Q:T\longrightarrow T.
\]
After passing to the algebraic closure $\bar K$ of $K$ and
identifying $T_{\bar K}$ with $J_{\bar K}$ by choosing an origin,
this action becomes translation by $Q_{\bar K}$ for the elliptic
curve group law. The resulting action on $T_{\bar K}$ is independent
of this choice; see \cite[Chapter~III, Proposition~3.4 and
Chapter~X, Theorem~3.8]{Sil09}.

Since \(\pi:X\to \Pp^1\) is relatively minimal, \(t_Q\) extends to an automorphism of \(X\) over \(\Pp^1\). These extensions define an injective homomorphism
\begin{equation}\label{eq:t_Q}
   J(K)\hookrightarrow\operatorname{Aut}(X/\Pp^1),
\qquad Q\longmapsto t_Q. 
\end{equation}

Define
\[
    P_F\coloneqq \{L\in\operatorname{Pic}(X)\mid c_1(L)\cdot F=0\}
\]
and let 
\[\theta: P_F\to
       \operatorname{Pic}^0(T), \quad L \mapsto L_\eta
       \] be the natural map. Set
\begin{equation}\label{eq:translation-subgroup-G}
    G\coloneqq\operatorname{Im}(\theta)\subseteq J(K),
\end{equation}
where $J(K)$ is the set of $K$-points.
Since $q(X)=0$, we have $\operatorname{Pic}^0(X)=0$, and therefore
$\operatorname{Pic}(X)=\NS(X)$ is a finitely generated abelian group.  Hence
$G$ is finitely generated. In particular,
\begin{equation}\label{eq:GmoddG-finite}
    G/dG \text{~is finite for every~}d>0.
\end{equation}

Fix $d>0$ and suppose that there is a horizontal integral curve of fiber degree
$d$.  Choose one such curve $C_0$.  For every horizontal integral curve $C$ of
the same degree define
\begin{equation}\label{eq:lambda-d-no-section}
    \lambda_d(C):=
    \bigl[\Oo_T(C_\eta-(C_0)_\eta)\bigr]\in G.
\end{equation}

We claim that
\begin{equation}\label{eq:translation-degree-d-formula}
    \lambda_d(t_Q(C))=\lambda_d(C)+dQ.
\end{equation}
Since the natural map $J(K)\to J(\bar K)$ is injective, it suffices to show the equation over the algebraic closure $\bar K$ of $K$. Write
$\bar T=T\times_K\bar K$ and $\bar Q\in J(\bar K)$ for the
image of $Q$. The base change of $t_Q$ is $t_{\bar Q}$.
Choose an origin $O\in\bar T(\bar K)$, giving the identification
\[
    \iota_O:\bar T\xrightarrow{\sim}J_{\bar K},
    \qquad p\longmapsto[\mathcal O_{\bar T}(p-O)].
\]
As mentioned above, $t_{\bar Q}$ is compatible with the elliptic curve group law under this identification, hence we have
\[
    \iota_O(t_{\bar Q}(p))=\iota_O(p)+\bar Q.
\]
This implies that
\begin{equation}\label{eq:bar Q}
        [\mathcal O_{\bar T}(t_{\bar Q}(p)-p)]=[\mathcal O_{\bar T}(t_{\bar Q}(p)-O)]-[\mathcal O_{\bar T}(p-O)]
    =\iota_O(t_{\bar Q}(p))-\iota_O(p)
    =\bar Q.
\end{equation}
Write
\[
    D_C\coloneqq (C_\eta)_{\bar K}=\sum_i n_i(p_i),
    \qquad \sum_i n_i=d,
    \qquad D_0\coloneqq ((C_0)_\eta)_{\bar K}.
\]
We have
\[
    ((t_Q(C))_\eta)_{\bar K}
    =(t_{\bar Q})(D_C)
    =\sum_i n_i(t_{\bar Q}(p_i)).
\]
Writing $\overline{\lambda_d(C)}$ for the image of
$\lambda_d(C)$ in $J(\bar K)$, by \eqref{eq:bar Q}, we obtain
\[
\begin{aligned}
    \overline{\lambda_d(t_Q(C))}
        -\overline{\lambda_d(C)}
    &=
    [\mathcal O_{\bar T}((t_{\bar Q})(D_C)-D_0)]
        -[\mathcal O_{\bar T}(D_C-D_0)]\\
    &=
    [\mathcal O_{\bar T}((t_{\bar Q})(D_C)-D_C)]\\
    &=
    \sum_i n_i
    [\mathcal O_{\bar T}(t_{\bar Q}(p_i)-p_i)]\\
    &=d\bar Q.
\end{aligned}
\]
This proves \eqref{eq:translation-degree-d-formula}.

\medskip
\noindent
\textit{Step 3: reducible fibers contribute only finitely many additional data.}

Let
\[
    \Theta_1,\ldots,\Theta_s
\]
be all irreducible components of all fibers whose reduced support is
reducible.  There are only finitely many such curves.  If a scheme-theoretic
fiber is
\[
    F_p=\sum_j n_{p,j}\Theta_{p,j},
    \qquad n_{p,j}>0,
\]
then for any horizontal curve $C$ of degree $d$,
\begin{equation}\label{eq:component-intersection-profile}
    d=C\cdot F_p
      =\sum_j n_{p,j}(C\cdot\Theta_{p,j}).
\end{equation}
Each $C\cdot\Theta_{p,j}$ is a nonnegative integer, and hence
\[
    0\leq C\cdot\Theta_{p,j}
    \leq\left\lfloor\frac d{n_{p,j}}\right\rfloor.
\]
Thus, for fixed $d$, the vector
\begin{equation}\label{eq:fiber-component-profile}
    \mathbf v(C):=(C\cdot\Theta_1,\ldots,C\cdot\Theta_s)
\end{equation}
has only finitely many possible values.

Next, we claim that the generic restriction together with this
intersection vector determines a negative curve uniquely.
To be precise, let $C,C'\subset X$ be horizontal integral curves satisfying
\[
    C\cdot F=C'\cdot F=d>0,
    \qquad
    C^2=(C')^2=b<0.
\]
Assume that
\[
    \Oo_T(C_\eta)\simeq\Oo_T(C'_\eta)
\]
and
\[
    C\cdot\Theta_i=C'\cdot\Theta_i
    \qquad\text{for }1\leq i\leq s.
\]
Then $C=C'$.

\noindent Indeed, equality of the restrictions to the generic fiber gives
\begin{equation}\label{eq:vertical-difference-general}
    C-C'\sim V
\end{equation}
for a vertical integral divisor $V$.  We have
\[
    V\cdot F=0,
    \qquad
    V\cdot\Theta_i=0\quad\text{for all }i.
\]
Every irreducible component of the vertical divisor \(V\) is either a component \(\Theta_i\) of a reducible fiber, or the reduced support \(E\) of a fiber whose reduced support is irreducible. In the first case \(V\cdot\Theta_i=0\) by assumption, while in the second case \(mE\sim F\) for some \(m>0\), and hence
\[
m(V\cdot E)=V\cdot F=0.
\]
Thus \(V\) is orthogonal to every irreducible component of its support. We obtain
\[
V^2=0.
\]

Choose an ample divisor $H$ and set
\[
    W:=V-\frac{V\cdot H}{F\cdot H}F.
\]
Then $W\cdot H=0$ and $W^2=0$.  By the Hodge index theorem, $W\equiv0$.
Thus $V\equiv\beta F$ for some $\beta\in\Rr$.  Since $C^2=(C')^2$,
\[
    0=(C-C')\cdot(C+C')=2\beta d,
\]
so $\beta=0$.  Hence $C\equiv C'$.  If $C\neq C'$, then
\[
    0\leq C\cdot C'=C^2=b<0,
\]
a contradiction.  This proves the claim.

Combining the claim with \eqref{eq:GmoddG-finite}, \eqref{eq:translation-degree-d-formula} and
\eqref{eq:fiber-component-profile}, we obtain the following conclusion:
for fixed $d>0$ and $b<0$, the group $G$ has only finitely many orbits on
horizontal integral curves satisfying
\[
    C\cdot F=d,\qquad C^2=b.
\]
By \eqref{eq:uniform-degree-bound-general-genus-one} and
\eqref{eq:elliptic-rational-square-general}, $G$ therefore has only finitely
many orbits on the set of all horizontal smooth rational curves.

\medskip
\noindent
\textit{Step 4: negative-definite configurations.}

Finiteness of the orbits of individual curves does not by itself imply
finiteness of the orbits of configurations.  We now control the relative
positions of the curves.

Choose a finite set $\mathcal R$ of representatives for the $G$-orbits of
horizontal smooth rational curves.  Let
\[
    \Gamma=C_1+\cdots+C_r
\]
be a nonempty reduced configuration of smooth rational curves whose
intersection matrix is negative definite.  If all components are vertical,
there are only finitely many possibilities by Step~1.  Otherwise, after
applying a translation to the whole configuration, we may suppose that one
horizontal component is
\[
    A\in\mathcal R.
\]
Put $d_A\coloneqq A\cdot F$.

Let $D\neq A$ be another horizontal component, and write
\[
    d\coloneqq D\cdot F,
    \qquad
    u\coloneqq A\cdot D.
\]
Then $1\leq d\leq M_X$ and $u\geq0$.  Since the intersection matrix of $\Gamma$ is negative definite,
\begin{equation}\label{eq:two-curve-negative-definite-bound}
    u^2<A^2D^2=(2+ad_A)(2+ad).
\end{equation}
Thus, for fixed $A$, only finitely many pairs $(d,u)$ can occur.

We claim that, for fixed $A,d,u$, only finitely many numerical classes of $D$
can occur.  Let $\NS(X)$ be the N\'eron-Severi group and let
\[
    L\coloneqq \NS(X)/\NS(X)_{\mathrm{tors}}.
\]
Consider
\[
    U\coloneqq \operatorname{span}_{\Rr}\{[F],[A]\}\subset L_{\Rr}.
\]
The intersection matrix on $U$ is
\[
    \begin{pmatrix}
        0&d_A\\
        d_A&A^2
    \end{pmatrix},
\]
whose determinant is $-d_A^2<0$.  Thus $U$ has signature $(1,1)$ and, by the
Hodge index theorem, $U^\perp$ is negative definite.

The two conditions
\[
    D\cdot F=d,
    \qquad
    D\cdot A=u
\]
uniquely determine the $U$-component of $[D]$.  Explicitly, write
\[
    [D]=x_0+w,
    \qquad w\in U^\perp,
\]
where
\begin{equation}\label{eq:x0-lattice-general}
    x_0=
    \frac d{d_A}[A]
    +\frac{u-(d/d_A)A^2}{d_A}[F].
\end{equation}
Since $D$ is smooth and rational,
\[
    D^2=-2-ad
\]
is fixed.  Hence $w^2=D^2-x_0^2$ is fixed.  A fixed-square locus in the
negative-definite real vector space $U^\perp$ is compact (or empty), while
$L$ is a lattice.  Therefore only finitely many classes $[D]\in L$ occur.
Moreover, a numerical class of negative square contains at most one integral
curve.  Hence, for fixed $A,d,u$, there are only finitely many possibilities
for $D$.

It follows that, after normalizing one horizontal component into
$\mathcal R$, every other horizontal component belongs to a fixed finite
collection of curves.  The vertical components also belong to a fixed finite
collection.  Consequently, $G$ has only finitely many orbits on the set of all
negative-definite reduced configurations of smooth rational curves.

\medskip
\noindent
\textit{Step 5: finiteness of log canonical models.}

Let
\[
    f:X\longrightarrow Z
\]
be a two-dimensional log canonical model of a klt pair on $X$.  As in the
proof of Proposition~\ref{prop:finite-negative-curves}, the map $f$ is a
birational contraction morphism, every irreducible component of
$\operatorname{Exc}(f)$ is a smooth rational curve, and the intersection
matrix of the reduced exceptional locus is negative definite.  By Step~4,
up to the action of $G\subset\operatorname{Aut}(X)$, there are only finitely
many possibilities for $\operatorname{Exc}(f)_{\mathrm{red}}$.  Precomposing
$f$ by an automorphism of $X$ does not change the isomorphism class of its
target, while the rigidity lemma \cite[Lemma~1.15]{Deb01} shows that a
birational contraction is determined, up to isomorphism of the target, by the
curves that it contracts.  Thus $\CC^2(X)$ is finite.

The finiteness of $\CC^1(X)$ follows from Theorem~\ref{thm:curve-targets}, and
there is at most one possible zero-dimensional target.  Hence $\CC(X)$ is finite.
\end{proof}

\section{Surface counterexamples for finiteness of log canonical models}\label{sec:surface-counterexamples}

In this section, we show that outside the cases covered in Section~\ref{sec:positive}, log canonical models could be infinite for non-minimal surfaces.

\subsection{Criteria for constructing distinct log canonical models}

We explain how to realize the contractions of negative rational
curves as log canonical models and how to distinguish their targets.

\begin{proposition}\label{prop:realize-negative-curve}
Let $X$ be a smooth projective surface and let $C\simeq\Pp^1$ satisfy $C^2=-q<0$. There is an effective $\Qq$-divisor $\Delta_C$ such that $(X,\Delta_C)$ is klt, $K_X+\Delta_C$ is big, and its log canonical model is the contraction of $C$.
\end{proposition}

\begin{proof}
Assume first that $q\geq2$ and set $b\coloneqq1-2/q\in[0,1)$. Adjunction gives $(K_X+bC)\cdot C=0$. For an ample divisor $H$, the divisor
\[
 P\coloneqq H+\frac{H\cdot C}{q}C
\]
vanishes on $C$, and is positive on every other irreducible curve. Hence, $P$ is nef. Thus $\Rr_{\geq0}[C]$ is an extremal ray of ${\NE}(X)$. For $0<\varepsilon\ll1$, it lies in the negative part of ${\NE}(X)$ for the klt pair $(X,(b+\varepsilon)C)$. The contraction theorem gives the contraction $f_C\colon X\to Y_C$ with $\Exc(f_C)=C$, and $Y_C$ is a $\Qq$-factorial surface. 

Writing $K_X=f_C^*K_{Y_C}+aC$ and intersecting with $C$ gives $a=-b$, hence $K_X+bC=f_C^*K_{Y_C}$. Choose an ample divisor $A_C$ on $Y_C$ such that $A_C-K_{Y_C}$ is ample. For a sufficiently divisible $m\geq2$, take a general
\[
 G_C\in|m(A_C-K_{Y_C})|
\]
that does not pass through the point $f_C(C)$. Its strict transform $B_C$ is smooth and disjoint from $C$. Then
\[
 \Delta_C\coloneqq bC+\frac1mB_C,
 \qquad K_X+\Delta_C\sim_{\Qq}f_C^*A_C,
\]
so $f_C$ is the required log canonical model.

If $q=1$, let $f_C$ be the blow-down of $C$. Choose an ample divisor $A_C$ such that $A_C-K_{Y_C}$ is very ample, take a general $G_C\in|m(A_C-K_{Y_C})|$ with $m\geq2$ that does not pass through $f_C(C)$, and set $\Delta_C\coloneqq m^{-1}f_C^*G_C$. Then
\[
 K_X+\Delta_C\sim_{\Qq}f_C^*A_C+C.
\]
As $C$ is the exceptional curve, we have $R(X, m(K_X+\Delta_C))=R(X, m(f_C^*A_C))$ for a sufficiently divisible $m \in \Zz_{>0}$. Thus, $f_C$ is still the log canonical model of $(X, \De_C)$.
\end{proof}

The following result will be used to distinguish log canonical models.

\begin{lemma}\label{lem:lifting-minimal-resolutions}
Let $f_i\colon X\to Y_i$ for $i=1,2$ be birational morphisms from a
smooth projective surface to normal projective surfaces such that
\[
\Exc(f_i)=C_i,
\]
where $C_i$ is a smooth rational curve with $C_i^2\leq-2$. If $Y_1\simeq Y_2$, then there exists $\tau \in \Aut(X)$ such that $\tau(C_1)=C_2$.
\end{lemma}

\begin{proof}
Set $y_i\coloneqq f_i(C_i)$. Since $\Exc(f_i)=C_i$, the morphism $f_i$
is an isomorphism away from $C_i$. Thus $Y_i$ is smooth away from
$y_i$. If $y_i$ were smooth, the factorization theorem for birational
morphisms between smooth surfaces would imply that $\Exc(f_i)$ contains
a $(-1)$-curve, contradicting $C_i^2\leq-2$. Hence $y_i$ is the unique
singular point of $Y_i$. Moreover, $f_i$ is the minimal resolution of
$Y_i$, because its exceptional locus contains no $(-1)$-curve.
Therefore, by the uniqueness of the minimal resolution, any isomorphism
$Y_1\simeq Y_2$ lifts to an automorphism $\tau\in\Aut(X)$ satisfying
$\tau(C_1)=C_2$.
\end{proof}

Combining the preceding results, we obtain the following construction,
which eliminates the automorphisms that could identify the contractions
of infinitely many negative curves.

\begin{proposition}\label{prop:very-general-blowup}
Let \(S\) be a smooth projective surface such that \(K_S\) is nef and
\(\Aut^0(S)\) is trivial. Suppose that \(S\) contains a sequence of pairwise
distinct smooth rational curves
\[
C_i\subset S,\qquad i\in\mathbb N,
\]
such that \(C_i^2\leq -2\). Let
\[
\mu\colon X\coloneqq\operatorname{Bl}_pS\longrightarrow S
\]
be the blow-up of a very general point $p\in S$. Then
\[
\Aut(X)=\{\Id_X\},
\]
and $\CC(X)$ is infinite.
\end{proposition}

\begin{proof}
For every non-identity \(g\in\Aut(S)\), the fixed locus \(\operatorname{Fix}(g)\) is a
proper closed subset of \(S\). Since \(\Aut^0(S)\) is trivial, $\Aut(S)$ consists of at most countably many elements. Take
\begin{equation}\label{eq:very-general-p}
p\notin
\left(\bigcup_{i\in\mathbb N}C_i\right)
\cup
\left(\bigcup_{\Id_S\neq g\in\Aut(S)}\operatorname{Fix}(g)\right),
\end{equation}
which lies in the complement of at most countably many proper closed subsets.

Let \(E\) be the exceptional curve of the blow-up \(\mu\colon X\coloneqq\operatorname{Bl}_pS\to S\). As \(K_S\) is nef, \(E\) is the
unique \((-1)\)-curve on \(X\). Since every automorphism \(g\in\Aut(X)\) preserves \(E\), by the rigidity lemma, the automorphism \(g\) descends along \(\mu\) to an
automorphism \(\bar g\in\Aut(S)\) satisfying
\[
\bar g\circ\mu=\mu\circ g.
\]
Since \(\bar g(p)=p\), the choice of \(p\) in
\eqref{eq:very-general-p} implies that \(\bar g=\Id_S\). It follows that
\(g=\Id_X\). Thus
\[
\Aut(X)=\{\Id_X\}.
\]

Since \(p\notin C_i\) for every \(i\), the strict transforms of the curves
\(C_i\), which we denote by the same symbols, remain pairwise distinct and
satisfy
\[
C_i^2\leq-2.
\]
By Proposition~\ref{prop:realize-negative-curve}, for every \(i\) there is a klt pair
whose log canonical model
\[
f_i\colon X\longrightarrow Y_i
\]
contracts exactly \(C_i\).

If \(Y_i\simeq Y_j\), then Lemma~\ref{lem:lifting-minimal-resolutions} gives an
automorphism of \(X\) carrying \(C_i\) to \(C_j\). Since \(\Aut(X)\) is
trivial, this implies \(C_i=C_j\), and hence \(i=j\). Therefore the varieties
\(Y_i\) are pairwise non-isomorphic, so \(\CC(X)\) is infinite.
\end{proof}

\subsection{Elliptic surfaces with infinitely many sections}\label{subsec: elliptic surfaces}

  The following construction serves as the basis for the two-dimensional
counterexamples in this section and the three-dimensional counterexample
in the next section. The same notation will be used throughout the rest of the paper.

Choose a rational elliptic surface with a section
\[
r\colon S\longrightarrow\Pp^1
\]
with twelve singular fibers, all of type \(I_1\). Such a surface exists by
\cite[Corollary~10.3]{SS10}. Moreover, every rational elliptic surface is
obtained by blowing up the base points of a cubic pencil; see
\cite[Proposition~8.1]{SS10}.

Let \(O\) be the zero section and let \(F\) denote the class of a fiber.
By the canonical bundle formula
\cite[Theorem~6.8]{SS10}, there exists a line bundle
\(\mathcal L\) on \(\Pp^1\) with
\(\deg\mathcal L=-\chi(\Oo_S)=-1\) such that
\begin{equation}\label{eq:canonical-rational-elliptic}
\omega_S
\simeq
r^*\bigl(\omega_{\Pp^1}\otimes\mathcal L^{-1}\bigr)\simeq r^*\Oo_{\Pp^1}(-1).
\end{equation}
Equivalently, we have \(K_S\sim-F\).

Let $K\coloneqq\Cc(\Pp^1)$ be the field of rational functions of the base and let $E/K$ be the generic fiber of $r$, with origin induced by $O$. The Mordell-Weil group $\MW(r)$ is the group $E(K)$ of $K$-points of $E$. Equivalently, it is the group of sections of $r$, with addition defined fiberwise. Its height pairing gives the Mordell-Weil lattice modulo torsion.

Since all fibers of \(r\) are irreducible, \cite[Corollary~8.4]{SS10} identifies the group of sections with the $E_8$ lattice. With the height-pairing convention of \cite[\S~11.7]{SS10}, the Mordell-Weil lattice is the positive-definite $E_8$ lattice. In particular,
$
\operatorname{rank}\MW(r)=8.
$
Fix a section \(P\) of infinite order.

For $d\in\{2,3\}$, choose a degree-$d$ morphism
\[
 \gamma_d\colon\Pp^1\to\Pp^1
\]
whose branch locus is disjoint from the discriminant locus of $r$. Let
\[
\begin{tikzcd}
S_d \arrow[r,"\eta_d"] \arrow[d,"\pi_d"']&
S \arrow[d,"r"]\\
\Pp^1 \arrow[r,"\gamma_d"']&
\Pp^1
\end{tikzcd}
\]
be the resulting base change.

Since \(\gamma_d\) is \'etale over the
discriminant locus, the surface \(S_d\) is smooth and \(\pi_d\) has exactly
\(12d\) singular fibers, all of type \(I_1\); see
\cite[\S~5.2]{SS10}. In particular, \(\pi_d\) is relatively minimal.

Each fiber of type \(I_1\) has Euler number \(1\). By \cite[Theorem~6.10 and Corollary~6.11]{SS10},
\[
e(S_d)=12d, \quad\chi(\Oo_{S_d})=\frac{e(S_d)}{12}=d.
\]

By the canonical bundle formula, there exists a line bundle
\(\mathcal L_d\) on \(\Pp^1\) with
\[\deg\mathcal L_d=-\chi(\Oo_{S_d})=-d\] such that

\begin{equation}\label{eq:canonical-Sd}
\omega_{S_d}
\simeq
\pi_d^*
\bigl(\omega_{\Pp^1}\otimes\mathcal L_d^{-1}\bigr)\simeq
\pi_d^*\Oo_{\Pp^1}(d-2).
\end{equation}

Let \(K_d/K\) be the function field extension induced by \(\gamma_d\). The
injection
\[
E(K)\hookrightarrow E(K_d)
\]
shows that the pullback section of \(P\), again denoted by \(P\), has
infinite order. Thus
\[
C_n\coloneqq nP,\qquad n\in\Zz,
\]
are pairwise distinct sections of \(\pi_d\). Here $nP$ is the Mordell-Weil multiple: for $n>0$ it is the fiberwise sum of $n$ copies of $P$, $0P=O$, and for $n<0$ it is the inverse of $(-n)P$.

By adjunction and \eqref{eq:canonical-Sd}, 
\begin{equation}\label{eq:section-square}
C_n^2=-d
\end{equation}
for every \(n\in\Zz\).

For both $d=2$ and $d=3$, \cite[Theorem~6.12]{SS10} gives
\begin{equation}\label{eq:pic0-Sd}
 \Pic^0(S_d)\simeq\Pic^0(\Pp^1)=0.
\end{equation}

If \(d=2\), then \eqref{eq:canonical-Sd} gives \(\omega_{S_2}\simeq \Oo_{S_2}\).
Thus \(q(S_2)=0\), and \(S_2\) is a K3 surface.

If \(d=3\), then
\[
\omega_{S_3}\simeq\pi_3^*\Oo_{\Pp^1}(1).
\] 
Therefore, 
$
\kappa(S_3)=1
$
and $S_3$ is a properly elliptic surface with \(\pi_3\) the Iitaka fibration.

The surfaces $S_d$ are non-ruled and have $q(S_d)=0$ by \eqref{eq:pic0-Sd}. Hence $\Aut^0(S_d)=0$ by \cite[Corollary~4.7]{Han87}, and $\Aut(S_d)$ is countable.

\begin{proof}[Proof of Theorem~\ref{thm:surface-counterexamples}]

For $\kappa=0$ or $1$, apply Proposition~\ref{prop:very-general-blowup} to $S_d$ for
$d=2,3$. The canonical divisors of $S_d$ are nef by
\eqref{eq:canonical-Sd}. Since the surfaces $S_d$ are non-ruled and
satisfy $q(S_d)=0$, \cite[Corollary~4.7]{Han87} gives
$\Aut^0(S_d)=0$. Moreover, \eqref{eq:section-square} supplies
infinitely many distinct smooth rational curves with
self-intersection at most $-2$. As the Kodaira dimension is unchanged by a blow-up, this produces surfaces of Kodaira dimension $0$ and $1$ with infinitely many log canonical models.

\medskip

For $\kappa=-\infty$, we return to the rational elliptic surface 
\[r\colon S\to\Pp^1\] chosen at the beginning of Section~\ref{subsec: elliptic surfaces}, and retain the notation $O$, $F$, and $P$. Put
\[
 C_n\coloneqq nP\qquad(n\in\Zz).
\]
Since $K_S\sim-F$, where $F$ is a fiber, adjunction gives
\[
 -2=(K_S+C_n)\cdot C_n=-1+C_n^2.
\]
Thus every $C_n$ is a $(-1)$-curve.

Choose a closed point $b\in\Pp^1$ such that the fiber $F_b\coloneqq r^{-1}(b)$ is smooth
and the point
\[
P(b)\coloneqq P\cap F_b\in F_b
\]
is non-torsion with respect to the group law on $F_b$ with origin
$O(b)\coloneqq O\cap F_b$. Indeed, on the smooth locus of $r$, the locus where $P(t)$ is torsion is
contained in
\[
\bigcup_{m\in \Zz\backslash\{0\}}\{t\mid (mP)(t)=O(t)\}.
\]
Each set in this union is finite, since $mP$ and $O$ are distinct
sections. We may choose $b$ outside that union and the finite discriminant locus. Note that $C_n \cap F_b=\{nP(b)\}$ and the set $\{nP(b)\mid n\in\Zz\}$ is countable. Choose
\[
 p\in F_b\setminus\{nP(b)\mid n\in\Zz\},
\]
and let
\[
 \mu\colon X\coloneqq\operatorname{Bl}_pS\to S.
\]
Then $X$ is rational, and the strict transforms of the $C_n$, which we
continue to denote by $C_n$, remain $(-1)$-curves.

Let
\[
 h\colon S\to Z
\]
be the contraction of the zero section, and denote its image by $z_0\coloneqq h(O)$.
Since $O$ is a $(-1)$-curve, its contraction produces a smooth surface
$Z$, and $z_0=h(O)$ is a smooth point of $Z$. We claim that $Z$ is a del Pezzo surface of degree one. Since
\[
 K_S=h^*K_Z+O
 \quad\text{and}\quad
 K_S\sim-F,
\]
one has
\[
 h^*(-K_Z)\sim F+O.
\]
Note that all singular fibers of $S$ are irreducible. Therefore $(F+O)\cdot\Gamma>0$ for every irreducible curve $\Gamma\neq O$, while 
\begin{equation}\label{eq: intersections}
  (F+O)\cdot O=0,\quad (-K_Z)^2=(F+O)^2=1.
\end{equation} Thus $-K_Z$ is ample by the Nakai-Moishezon criterion.

Because $Z$ is rational, $\chi(\Oo_Z)=1$. Riemann-Roch and Kodaira vanishing applied to $-K_Z=K_Z+(-2K_Z)$ give
\[
h^0(Z,-K_Z)=1+K_Z^2=2,
\]
so \( |-K_Z| \) is a pencil. 

Since $\Oo_S(F)=r^*\Oo_{\Pp^1}(1)$, we have $|F|=\{F_t\mid t\in\Pp^1\}$ and
\[
H^0\bigl(\Pp^1,\Oo_{\Pp^1}(1)\bigr)
\simeq H^0\bigl(S,F\bigr)
\hookrightarrow H^0\bigl(S,F+O\bigr)
\simeq H^0\bigl(S,h^*(-K_Z)\bigr).
\]
Since
\[
h^0\bigl(\Pp^1,\Oo_{\Pp^1}(1)\bigr)
=2
=h^0\bigl(S,h^*(-K_Z)\bigr),
\]
we have
\[
h^*|-K_Z|
=|F|+O.
\]
Pushing forward by $h$, we obtain
\[
|-K_Z|=h_*|F|=\{h(F_t)\mid t\in\Pp^1\}.
\]
Every member contains $z_0=h(O)$, and images of distinct fibers are disjoint away
from $z_0$. 
Let
\[
 D\coloneqq h(F_b)\in|-K_Z|.
\]
Then, we have
\[
 h^*D=F_b+(F_b\cdot O)O=F_b+O,
\]
so $\operatorname{mult}_{z_0}(D)=F_b\cdot O=1$. Hence the curve $D$ is
smooth at $z_0$, the restriction $h|_{F_b}\colon F_b\to D$ is an isomorphism, and
$D$ is a smooth elliptic curve. 

Let \(K=\Cc(\Pp^1)\) and let \(E/K\) be the generic fiber of \(r\),
viewed as an elliptic curve over \(K\). Translation by
\(-nP\in E(K)\) defines an automorphism of \(E/K\), and hence a
birational map
\[
t_{-nP}\colon S\dashrightarrow S
\]
over \(\Pp^1\). Hence \(t_{-nP}\) extends to an automorphism of the relatively minimal surface $S$ over $\Pp^1$, and \(t_{-nP}(C_n)=O\). Define
\[
 p_n\coloneqq t_{-nP}(p)\in F_b,
 \qquad
 q_n\coloneqq h(p_n)\in D.
\]
The points \(p_n\) are pairwise distinct because \(P(b)\) is
non-torsion. The choice of \(p\) gives \(p_n\notin O\), so the points
\(q_n\) are pairwise distinct and different from \(z_0\).

After applying $t_{-nP}$, the contraction of $C_n$ on $X$ becomes the contraction of $O$ on $\operatorname{Bl}_{p_n}S$. Since $p_n\notin O$, blowing up $p_n$ and contracting $O$ commute. So we have a commutative diagram
\[
\begin{tikzcd}[column sep=large]
 X=\operatorname{Bl}_pS
   \arrow[r,"\sim"]
   \arrow[d,"\mu"']
   \arrow[rr,bend left=18,"f_n"]
 & \operatorname{Bl}_{p_n}S
   \arrow[r]
   \arrow[d]
 & Y_n=\operatorname{Bl}_{q_n}Z
   \arrow[d,"\sigma_n"] \\
 S
   \arrow[r,"t_{-nP}"']
 & S
   \arrow[r,"h"']
 & Z
\end{tikzcd}
\]
Here $f_n$ contracts $C_n$, and the diagram identifies its target with $Y_n=\operatorname{Bl}_{q_n}Z$.
By Proposition~\ref{prop:realize-negative-curve}, the contraction
\(X\to Y_n\) is the log canonical model of a klt pair on \(X\).

It remains to distinguish infinitely many of the surfaces \(Y_n\).
Using the above identification, write
\[
 \sigma_n\colon Y_n\to Z
\]
for the blow-up morphism, let $E_n$ be its exceptional curve, and let $D_n$
be the strict transform of $D$. This means that
\[
 |-K_{Y_n}|
 =
 \left\{
 \sigma_n^*B-E_n
 \,\middle|\,
 B\in|-K_Z|,\ q_n\in\operatorname{Supp}(B)
 \right\}.
\]
Exactly one member of $|-K_Z|$ passes through $q_n$, namely $D$; this holds for every $n\in\Zz$ because $q_n\ne z_0$. Since
$D$ is smooth at $q_n$, one has 
\begin{equation}\label{eq: D_n}
  D_n =\sigma_n^*D-E_n.
\end{equation} The curve $D_n$ is irreducible, and hence
\begin{equation}\label{eq:unique-anticanonical}
 |-K_{Y_n}|=\{D_n\}.
\end{equation}
Consequently, every isomorphism $Y_n\to Y_m$ maps $D_n$ to $D_m$.

For any $D'\in|-K_Z|$ with $D'\neq D$, we have
\[
 D\cdot D'=(-K_Z)^2=1
\] by \eqref{eq: intersections}.
Since both $D$ and $D'$ pass through $z_0$, their intersection
divisor on $D$ is $D'|_D=\{z_0\}$. It follows that
\[
 \Oo_D(D)
 \simeq
 \Oo_D(D')
 \simeq
 \Oo_D(z_0).
\]

Since $D$ is smooth at $q_n$, one has
$\operatorname{mult}_{q_n}(D)=1$. Hence $D_n$ meets $E_n$ transversely at
one point, and
\[
 \sigma_n|_{D_n}\colon D_n\xrightarrow{\sim}D
\]
is an isomorphism. Under this isomorphism, the point ${D_n}\cap {E_n}$
corresponds to $q_n$. By \eqref{eq: D_n}, we have
\[
 N_{D_n/Y_n}\simeq\Oo_{D_n}(D_n)
 \simeq(\sigma_n|_{D_n})^*\Oo_D(D-q_n).
\]
Identifying $D_n$ with $D$ via $\sigma_n|_{D_n}$, we obtain
\begin{equation}\label{eq:normal-anticanonical}
 N_{D_n/Y_n}
 \simeq
 \Oo_D(D-q_n)
 \simeq
\Oo_D(z_0-q_n)
 \in\Pic^0(D).
\end{equation}
If $Y_n\simeq Y_m$, restriction to the unique anticanonical curves
$D_n$ and $D_m$ gives an isomorphism $D_n\simeq D_m$. Consequently,
there exists an automorphism $\alpha\in\Aut(D)$ satisfying
\[
 \alpha^*\Oo_D(z_0-q_m)\simeq\Oo_D(z_0-q_n).
\]
Use $z_0$ as the origin and write
$\alpha=t_a\circ u$, where $t_a$ is a translation and
\[u \in \Aut(D,z_0) \coloneqq \{\sigma \in \Aut(D) \mid \sigma(z_0)=z_0\}.\] As $D$ is an elliptic curve, the group $\Aut(D,z_0)$ is finite. Translations
act trivially on $\Pic^0(D)$, and
\[
 D\longrightarrow\Pic^0(D),
 \qquad
 q\longmapsto\Oo_D(z_0-q),
\]
is an isomorphism. Hence the above relation is equivalent to
$q_m=u(q_n)$. Thus, for fixed $n$, the possible points $q_m$ lie in the
finite orbit $\Aut(D,z_0)\cdot q_n$. Since the points $q_m$ are pairwise
distinct, only finitely many $m$ satisfy $Y_n\simeq Y_m$. Consequently, each
isomorphism class is represented by only finitely many members of
$\{Y_i\}_{i\in\Zz}$. It follows that infinitely many of the surfaces $Y_i$
are pairwise non-isomorphic.
\end{proof}

\section{A canonically polarized threefold with infinitely many log canonical models}\label{sec:threefold}

In this section, we show that the direct higher-dimensional analogue
of Theorem~\ref{thm:minimal surfaces} is false. In fact, there exists a smooth threefold with ample canonical divisor which admits infinitely many log canonical models.

We begin with a local computation which will be used to distinguish the contraction targets in the construction. 

\begin{lemma}\label{lem:embedding-dimension}
Let $V$ be a smooth complex threefold, let $Y$ be a normal complex threefold, and let $C\simeq\Pp^1$ be a curve on $V$. Suppose that
\[
 f\colon V\to Y
\]
is a projective birational contraction with $\Exc(f)=C$, and set $y\coloneqq f(C)$. Assume that the formal completion of $V$ along $C$ is isomorphic to the formal completion of the zero section in
\[
 \Tot_{\Pp^1}\bigl(\Oo_{\Pp^1}(-1)\oplus\Oo_{\Pp^1}(-d)\bigr)
\]
for some integer $d\geq1$. Then $y$ is the unique singular point of $Y$ with embedding dimension
\[
 \edim_yY=d+3.
\]
In particular, $d$ is determined by the isomorphism class of $Y$.
\end{lemma}

\begin{proof}
The exceptional curve $C$ is contracted to the point $y$. Since $f$ has connected fibers, $f^{-1}(y)=C$ set-theoretically. The restriction $V\setminus C\to Y\setminus\{y\}$ is an isomorphism. Thus $Y$ is smooth away from $y$.

We compute the completed local ring at $y$. After replacing $Y$ by an affine neighborhood of $y$ and $V$ by its inverse image, let $\mathcal I_C$ be the ideal sheaf of $C$. The theorem on formal functions \cite[III, Theorem~11.1]{Har77} gives
\[
 \widehat{\Oo}_{Y,y}
 \simeq
 \varprojlim_n
 H^0\!\left(V,\Oo_V/\mathfrak m_y^{n+1}\Oo_V\right).
\]
Since $f^{-1}(y)=C$ set-theoretically, the ideals $\mathfrak m_y\Oo_V$ and $\mathcal I_C$ have the same radical. Because $V$ is noetherian, there are integers $a,b>0$ with $\mathcal I_C^a\subset\mathfrak m_y\Oo_V$ and $(\mathfrak m_y\Oo_V)^b\subset\mathcal I_C$. Their powers are cofinal, so the two inverse limits agree. Consequently,
\begin{equation}\label{eq:formal-functions-C}
 \widehat{\Oo}_{Y,y}
 \simeq
 H^0\!\left(\widehat V_C,\Oo_{\widehat V_C}\right).
\end{equation}

Put
\[
 E_d\coloneqq\Oo_{\Pp^1}(-1)\oplus\Oo_{\Pp^1}(-d),
 \qquad
 X_d\coloneqq\Tot_{\Pp^1}(E_d),
\]
and let $Z_d\simeq\Pp^1$ denote the zero section. Define
\[
 R_d\coloneqq\bigoplus_{m\geq0}
 H^0\!\left(\Pp^1,\Sym^mE_d^\vee\right),
 \qquad
 \mathfrak m_d\coloneqq\bigoplus_{m>0}(R_d)_m.
\]
Under $X_d=\spec_{\Pp^1}(\Sym E_d^\vee)$, the ideal $\mathcal I_{Z_d}$ is the kernel of the natural projection \[\Sym E_d^\vee\to\Oo_{\Pp^1}.\] Thus
\[
 \mathcal I_{Z_d}^{n+1}=\bigoplus_{m\geq n+1}\Sym^mE_d^\vee,
\]
and
\[
 H^0\!\left(X_d,\Oo_{X_d}/\mathcal I_{Z_d}^{n+1}\right)
 =
 \bigoplus_{m=0}^n
 H^0\!\left(\Pp^1,\Sym^mE_d^\vee\right).
\]
By the assumption of formal isomorphism and \eqref{eq:formal-functions-C}, it follows that
\begin{equation}\label{eq:completed-local-scroll-ring}
 \widehat{\Oo}_{Y,y}
 \simeq
 \varprojlim_n\bigoplus_{m=0}^n(R_d)_m.
\end{equation}

The graded ring $R_d$ is generated by its degree-one part. Indeed, we have
\begin{equation}\label{eq: R_d}
   (R_d)_m
 =
 \bigoplus_{a+b=m}
 H^0\!\left(\Pp^1,\Oo_{\Pp^1}(a+bd)\right).
\end{equation}
Moreover, the multiplication maps
\[
 H^0(\Pp^1,\Oo(r))\otimes H^0(\Pp^1,\Oo(s))
 \longrightarrow H^0(\Pp^1,\Oo(r+s))
\]
are surjective for $r,s\geq0$. Consequently,
\[
 \mathfrak m_d^{n+1}
 =
 \bigoplus_{m\geq n+1}(R_d)_m,
\]
so the inverse limit in \eqref{eq:completed-local-scroll-ring} is the $\mathfrak m_d$-adic completion of $(R_d)_{\mathfrak m_d}$. Completion does not change the cotangent-space dimension, and therefore, by \eqref{eq: R_d},
\[
 \edim_yY=\dim_{\Cc}\mathfrak m_d/\mathfrak m_d^2
 =h^0(\Pp^1,\Oo(1))+h^0(\Pp^1,\Oo(d))=d+3.
\]
Since $d\geq1$, this is strictly larger than $\dim Y=3$, so $y$ is singular. Hence $y$ is the unique singular point of $Y$.
\end{proof}

Now we give the construction of a canonically polarized threefold
that admits infinitely many log canonical models.

\begin{proof}[Proof of Theorem~\ref{thm:intro-threefold}]
Let $r\colon S\to\Pp^1$ be the rational elliptic surface constructed at
the beginning of Section~\ref{subsec: elliptic surfaces}, and retain its
infinitely many $(-1)$-sections $C_n=nP$. Choose a very ample divisor
$H$ on $S$ such that $K_S+H$ is ample. After passing to a subsequence
and relabeling its members as $\{C_i\}_{i\geq1}$, we may assume that
\begin{equation}\label{eq:di-increasing}
  d_i\coloneqq H\cdot C_i
\end{equation}
is strictly increasing. Indeed, for a fixed integer $d$, all sections satisfying $H\cdot C_i=d$ have Hilbert polynomial $P_d(n)=dn+1$. Let $\operatorname{Hilb}_{P_d}(S)$ be the projective Hilbert scheme parametrizing subschemes of $S$ with this polynomial. Its tangent space at $[C_i]$ is
\[
 T_{[C_i]}\operatorname{Hilb}_{P_d}(S)
 \simeq H^0(C_i,N_{C_i/S})
 =H^0(\Pp^1,\Oo_{\Pp^1}(-1))=0
\]
by \cite[Theorem~1.1(b)]{Har10}. The local maximal ideal at $[C_i]$ therefore satisfies $\mathfrak m/\mathfrak m^2=0$; Nakayama's lemma gives $\mathfrak m=0$, so $[C_i]$ is a reduced isolated point. A noetherian scheme has only finitely many isolated points, so only finitely many $C_i$ have a fixed $H$-degree. Hence the degrees are unbounded.

As mentioned in Section~\ref{sec:preliminaries}, we use Grothendieck's quotient convention for projective bundles and set
\[
 p\colon W\coloneqq\Pp_S(\Oo_S\oplus\Oo_S(H))\to S,
 \qquad
 \xi\coloneqq c_1(\Oo_W(1)).
\]
The quotient $\Oo_S\oplus\Oo_S(H)\twoheadrightarrow\Oo_S$ defines a section
\[
 T\simeq S\subset W
\]
satisfying
\begin{equation}\label{eq:T-normal}
 \Oo_W(1)|_T\simeq\Oo_T,
 \qquad
 N_{T/W}\simeq\Oo_T(-H),
\end{equation}
where the normal-bundle formula follows from the relative Euler sequence. In the above expression and what follows, we write $\Oo_T(-H)$ for the pullback of $\Oo_S(-H)$ along the natural morphism $T \to S$.

The bundle $\Oo_S\oplus\Oo_S(H)$ is globally generated, and the tautological quotient
\[
 p^*(\Oo_S\oplus\Oo_S(H))\twoheadrightarrow\Oo_W(1)
\]
shows that $\xi$ is base-point-free.
Moreover,
\[
 p_*\Oo_W(6\xi)
 =\Sym^6(\Oo_S\oplus\Oo_S(H))
\]
contains $\Oo_S$ as a direct summand, and restriction to $T$ is the projection onto this summand. Thus, the natural map
\[
 H^0(W,\Oo_W(6\xi))\longrightarrow H^0(T,\Oo_T)
\]
is surjective. 
By Bertini's theorem, a general such section defines a smooth divisor
\[
 \mathcal B\in|6\xi|
\]
disjoint from $T$. Indeed, we have $\Oo_W(6\xi)|_T\simeq\Oo_T$.

Let 
\[\rho\colon V\to W\] be the double cover branched along $\mathcal B$. To be precise, let $s_{\mathcal B}\in H^0(W,\Oo_W(6\xi))$ define $\mathcal B$, and set
\[
 V\coloneqq\spec_W\bigl(\Oo_W\oplus\Oo_W(-3\xi)\bigr),
\]
where multiplication on the second summand is induced by $s_{\mathcal B}$. Locally this has equation $z^2=s_{\mathcal B}$. Since the branch divisor is smooth, $V$ is a smooth projective threefold.

The canonical divisor of the cover is
\begin{equation}\label{eq:KV-first}
 K_V=\rho^*(K_W+3\xi).
\end{equation}
The projective bundle formula gives $K_W=-2\xi+p^*(K_S+H)$, and hence
\begin{equation}\label{eq:KV}
 K_V=\rho^*\bigl(\xi+p^*(K_S+H)\bigr).
\end{equation}

The divisor inside the pullback is ample. Indeed, tensoring
$\Oo_S\oplus\Oo_S(H)$ with $\Oo_S(K_S+H)$ gives a natural identification
\[
 \Pp_S(\Oo_S\oplus\Oo_S(H))
 \simeq
 \Pp_S\bigl(\Oo_S(K_S+H)\oplus\Oo_S(K_S+2H)\bigr),
\]
and the tautological divisor becomes $\xi+p^*(K_S+H)$. Both summands on the right are ample, so their direct sum is ample; see \cite[Proposition~6.1.13]{Laz04II}. Hence $\xi+p^*(K_S+H)$, and therefore $K_V$, is ample.

Because $\mathcal B$ is disjoint from $T$, the double cover is finite étale in a neighborhood of $T$. To describe its restriction, note that $\Oo_W(-3\xi)|_T\simeq\Oo_T$ and $s_{\mathcal B}|_T=c\in\Cc^*$. Hence
\[
 \rho_*\Oo_V|_T\simeq\Oo_T[z]/(z^2-c)\simeq\Oo_T\times\Oo_T.
\]
Therefore, we have a disjoint union
\begin{equation}\label{eq:split-over-T}
 \rho^{-1}(T)=T_+\sqcup T_- \text{~such that~}
 \rho|_{T_\pm}\colon T_\pm\xrightarrow{\sim}T.
\end{equation}

Let $j\colon S\xrightarrow{\sim}T$ be the section of $p$, and fix $T_+$. For $q=p \circ \rho$, the restriction $q|_{T_+}=p\circ\rho|_{T_+}$ identifies $T_+$ with $S$; denote its inverse by $\iota_+\colon S\xrightarrow{\sim}T_+$. Since $\rho$ is étale along $T_+$, \eqref{eq:T-normal} gives
\begin{equation}\label{eq:Tplus-normal}
 N_{T_+/V}\simeq(q|_{T_+})^*\Oo_S(-H).
\end{equation}
To distinguish the copies of a section, put
\[
 C_i^T\coloneqq j(C_i)\subset T,
 \qquad
 \Gamma_i\coloneqq\iota_+(C_i)\subset T_+.
\]
Note that $C_i\subset S$, $C_i^T\subset T$, and $\Gamma_i\subset T_+$ are different curves related by the displayed isomorphisms. Since $N_{C_i/S}\simeq\Oo_{\Pp^1}(-1)$, we have $N_{\Gamma_i/T_+}\simeq\Oo_{\Pp^1}(-1)$. The normal-bundle sequence for the regular immersions $\Gamma_i\subset T_+\subset V$ (see \cite[\href{https://stacks.math.columbia.edu/tag/063N}{Tag~063N}]{Stacks}) is
\[
 0\longrightarrow N_{\Gamma_i/T_+}
 \longrightarrow N_{\Gamma_i/V}
 \longrightarrow\Oo_{\Pp^1}(-d_i)
 \longrightarrow0.
\]
Since
\[
 \Ext^1\bigl(\Oo_{\Pp^1}(-d_i),\Oo_{\Pp^1}(-1)\bigr)
 =H^1(\Pp^1,\Oo_{\Pp^1}(d_i-1))=0,
\]
the sequence splits, and hence
\begin{equation}\label{eq:normal-Ci-V}
 N_{\Gamma_i/V}
 \simeq
 \Oo_{\Pp^1}(-1)\oplus\Oo_{\Pp^1}(-d_i).
\end{equation}

We next construct a semiample divisor that contracts exactly $\Gamma_i$. Let
\[
 h_i\colon S\to Z_i
\]
be the contraction of the $(-1)$-curve $C_i$. Choose an ample Cartier divisor $A_i$ on $Z_i$ and put
\[
 P_i\coloneqq h_i^*A_i.
\]
Then $P_i$ is nef, and for an irreducible curve $\Gamma\subset S$ we have
\begin{equation}\label{eq:Pi-zero-curves}
 P_i\cdot\Gamma=0
 \quad\Longleftrightarrow\quad
 \Gamma=C_i.
\end{equation}

The morphisms used below fit into the commutative diagram
\[
\begin{tikzcd}[column sep=large,row sep=small]
T_+ \arrow[r,hook] \arrow[d,"q|_{T_+}"'] &
V \arrow[r,"\rho"] \arrow[d,"q"] &
W \arrow[d,"p"]\\
S \arrow[r,equal] \arrow[d,"h_i"'] &
S \arrow[r,equal] \arrow[d,"h_i"] &
S \arrow[d,"h_i"]\\
Z_i \arrow[r,equal] & Z_i \arrow[r,equal] & Z_i.
\end{tikzcd}
\]

Let $\Xi\coloneqq\rho^*\xi$, and define
\begin{equation}\label{eq:Mi}
 M_i\coloneqq\Xi+q^*(H+P_i)+T_+.
\end{equation}
The divisor
\[
 A_i'\coloneqq\Xi+q^*(H+P_i)
\]
is ample. Indeed, it is the finite pullback of the tautological divisor associated with
\[
 \Oo_S(H+P_i)\oplus\Oo_S(2H+P_i),
\]
whose two summands are ample. If an irreducible curve $\Gamma$ is not contained in $T_+$, then $T_+\cdot\Gamma\geq0$, and hence
\[
 M_i\cdot\Gamma\geq A_i'\cdot\Gamma>0.
\]
On the other hand, using $\Xi|_{T_+}=0$ and \eqref{eq:Tplus-normal}, we obtain
\begin{equation}\label{eq:Mi-restriction}
 \iota_+^*(M_i|_{T_+})
 =(H+P_i)-H
 =P_i.
\end{equation}
Thus $M_i$ is nef. It is also big, since $M_i=A_i'+T_+$ is the sum of an ample divisor and an effective divisor.

We claim that
\begin{equation}\label{eq:zero-face}
 \NE(V)\cap M_i^\perp=\Rr_{\geq0}[\Gamma_i].
\end{equation}
The inclusion $\supseteq$ follows from \eqref{eq:Mi-restriction} as $P_i\cdot C_i=0$. Let $\alpha\in\NE(V)$ satisfy $M_i\cdot\alpha=0$, and choose effective one-cycles $\theta_n$ whose numerical classes converge to $\alpha$. Write
\[
 \theta_n=\theta_n^++\theta_n^-,
\]
where every component of $\theta_n^+$ is contained in $T_+$ and no component of $\theta_n^-$ is contained in $T_+$. Since $M_i$ is nef and $M_i\cdot\theta_n\to0$, we have
\[
 0\leq A_i'\cdot\theta_n^-
 \leq M_i\cdot\theta_n^-
 \longrightarrow0.
\]
As $A_i'$ is ample, this implies $[\theta_n^-]\to0$ in $N_1(V)$. The pushforward $N_1(T_+)\to N_1(V)$ is injective because $q|_{T_+}\colon T_+\to S$ is an isomorphism and $q_*\circ(\iota_+)_*=\Id$. Hence the pushforward classes of $\theta_n^+$ converge to the pushforward of the class $\beta\coloneqq q_*\alpha\in\NE(S)$. By \eqref{eq:Mi-restriction},
\[
 P_i\cdot\beta=0.
\]
The zero face of $P_i=h_i^*A_i$ is generated by the unique curve contracted by $h_i$, namely $C_i$. Thus the corresponding class on $V$ lies in $\Rr_{\geq0}[\Gamma_i]$, proving \eqref{eq:zero-face}.

Adjunction gives $K_{T_+}=(K_V+T_+)|_{T_+}$. Under $T_+\simeq S$, equation \eqref{eq:Tplus-normal} says $T_+|_{T_+}=-(q|_{T_+})^*H$, and therefore
\[
 K_V|_{T_+}=(q|_{T_+})^*(K_S+H).
\]
Since $C_i$ is a $(-1)$-curve, $K_S\cdot C_i=-1$, and therefore
\[
 K_V\cdot\Gamma_i=d_i-1,
 \qquad
 T_+\cdot\Gamma_i=-d_i.
\]
Set
\[
 a_i\coloneqq1-\frac{1}{2d_i}\in(0,1).
\]
Then
\begin{equation}\label{eq:negative-pair-Ci}
 (K_V+a_iT_+)\cdot\Gamma_i=-\frac12.
\end{equation}
We claim that for every sufficiently large integer $t_i$ the divisor
\begin{equation}\label{eq:Li}
 L_i\coloneqq t_iM_i-(K_V+a_iT_+)
\end{equation}
is ample. To see this, fix an ample divisor $B$ on $V$ and consider the compact slice
\[
 \Sigma\coloneqq\{\alpha\in\NE(V)\mid B\cdot\alpha=1\}.
\]
By \eqref{eq:zero-face}, the zero locus of $M_i$ on $\Sigma$ is the class $\lambda_i[\Gamma_i]$, where $\lambda_i=(B\cdot\Gamma_i)^{-1}$. By \eqref{eq:negative-pair-Ci}, the divisor $-(K_V+a_iT_+)$ is positive at this point, and hence on a neighborhood $U$ of it in $\Sigma$. On the compact set $\Sigma\setminus U$, the function $M_i$ has a positive minimum, whereas $-(K_V+a_iT_+)$ is bounded below. It follows that $L_i$ is positive on all of $\Sigma$ for $t_i\gg0$. Kleiman's criterion then shows that $L_i$ is ample.

The pair $(V,a_iT_+)$ is klt. Since $M_i$ is nef Cartier and
\[
 t_iM_i-(K_V+a_iT_+)=L_i
\]
is ample, the base-point-free theorem \cite[Theorem~3.3]{KM98} shows that $M_i$ is semiample. Let
\[
 f_i\colon V\to Y_i
\]
be the contraction morphism induced by $M_i$. Then $Y_i$ is normal and
\[
 M_i\sim_{\Qq}f_i^*B_i
\]
for some ample $\Qq$-Cartier divisor $B_i$ on $Y_i$. Since $M_i$ is big,
$f_i$ is birational. By \eqref{eq:zero-face} and the preceding discussion,
$\Gamma_i$ is the only curve contracted by $f_i$; hence
\begin{equation}\label{eq:exceptional-locus-fi}
 \Exc(f_i)=\Gamma_i.
\end{equation}
Set $y_i\coloneqq f_i(\Gamma_i)$.

Choose a sufficiently divisible integer $m_i\geq2$ such that $m_iL_i$ is Cartier and very ample. Take a general member
\[
 G_i\in|m_iL_i|.
\]
By Bertini's theorem, $G_i$ is smooth and meets $T_+$ transversely. Put
\[
 \Delta_i\coloneqq a_iT_++\frac1{m_i}G_i.
\]
The support of $\Delta_i$ has simple normal crossings, so $(V,\Delta_i)$ is klt. Moreover,
\[
 K_V+\Delta_i
 \sim_{\Qq}
 K_V+a_iT_++L_i
 =t_iM_i.
\]
Thus
\[
 K_V+\Delta_i\sim_{\Qq}f_i^*(t_iB_i),
\]
where $t_iB_i$ is ample on $Y_i$. Therefore $f_i$ is the log canonical model of $(V,\Delta_i)$. In particular,
\[
 [Y_i]\in\CC^3(V).
\]

It remains to distinguish the targets $Y_i$. We identify the formal neighborhood of $\Gamma_i$ in $V$. The contraction $h_i$ maps the $(-1)$-curve $C_i$ to a smooth point, so locally it is the blow-up of that point. The standard morphism
\[
 \Tot_{\Pp^1}\Oo_{\Pp^1}(-1)=\operatorname{Bl}_0\mathbb A^2\longrightarrow\mathbb A^2
\]
contracts its zero section. Therefore, if $p_0\colon\Tot_{\Pp^1}\Oo_{\Pp^1}(-1)\to\Pp^1$ is the bundle projection, then we have the isomorphism of the formal completions
\begin{equation}\label{eq:formal-surface-model}
 \widehat S_{C_i}\simeq
 \widehat{\Tot_{\Pp^1}\Oo_{\Pp^1}(-1)}_{\Pp^1}.
\end{equation}

Let $\mathcal I_i$ be the ideal of $C_i$ in $S$ and let $C_i^{(n)}$ be defined by $\mathcal I_i^{n+1}$. Since
\[
 \mathcal I_i^n/\mathcal I_i^{n+1}\simeq\Oo_{\Pp^1}(n),
\]
the unit sequences for successive thickenings \cite[III, Exercise~4.6]{Har77}, together with vanishings \[H^1(\Pp^1,\Oo(n))=H^2(\Pp^1,\Oo(n))=0,\] show that $\Pic(C_i^{(n)})\to\Pic(C_i^{(n-1)})$ is an isomorphism. By \cite[II, Exercise~9.6]{Har77}, restriction along the canonical closed immersion $C_i\hookrightarrow\widehat S_{C_i}$ gives
\begin{equation}\label{eq:formal-picard}
 \Pic(\widehat S_{C_i})\simeq
 \varprojlim_n\Pic(C_i^{(n)})\simeq\Pic(C_i).
\end{equation}

Let
\[
\nu_i\colon
\widehat S_{C_i}
\xrightarrow{\sim}
\widehat{\Tot_{\Pp^1}\Oo_{\Pp^1}(-1)}_{\Pp^1}
\to
\Tot_{\Pp^1}\Oo_{\Pp^1}(-1)
\]
be the morphism induced by \eqref{eq:formal-surface-model}, where the
second arrow is the canonical morphism from the formal completion.
Since
\[
\Oo_S(-H)|_{C_i}\simeq\Oo_{\Pp^1}(-d_i),
\]
equations \eqref{eq:formal-surface-model} and
\eqref{eq:formal-picard} yield
\begin{equation}\label{eq:formal-line-bundle}
\Oo_S(-H)|_{\widehat S_{C_i}}
\simeq
\nu_i^*p_0^*\Oo_{\Pp^1}(-d_i).
\end{equation}

Let $e_0$ denote the homogeneous section of $\Oo_W(1)$ induced by the inclusion \[\Oo_S \to \Oo_S\oplus\Oo_S(H).\] The relative affine open set $D_+(e_0)\subset W$ containing $T$ is $\Tot_S\Oo_S(-H)$, with $T$ as its zero section. Locally, if $C_i$ is defined by $I\subset A$ and the line bundle is trivial, completing its total space along the zero section over $C_i$ gives
\[
 \widehat{A[z]}_{(I,z)}\simeq\widehat A_I[[z]].
\]
Thus the formal total space depends only on the formal base and the restricted line bundle. Using \eqref{eq:formal-surface-model} and \eqref{eq:formal-line-bundle}, we obtain
\begin{align}
 \widehat W_{C_i^T}
 &\simeq \widehat{\Tot_S\Oo_S(-H)}_{C_i}\notag\\
 &\simeq
 \widehat{\Tot_{\Pp^1}
 (\Oo_{\Pp^1}(-1)\oplus\Oo_{\Pp^1}(-d_i))}_{\Pp^1}.
 \label{eq:formal-neighborhood-W}
\end{align}

It remains to transfer this model through the double cover. Let $T^{(n)}$ be the $n$-th infinitesimal neighborhood of $T$ in $W$. Finite \'etale covers are invariant under nilpotent thickenings \cite[\href{https://stacks.math.columbia.edu/tag/0BQB}{Tag~0BQB}]{Stacks}. The split cover \eqref{eq:split-over-T} therefore extends uniquely and compatibly to
\[
 V\times_WT^{(n)}\simeq T^{(n)}\sqcup T^{(n)}
 \qquad(n\geq0).
\]
Passing to the inverse limit gives
\[
 \widehat V_{T_+\sqcup T_-}\simeq\widehat W_T\sqcup\widehat W_T,
 \qquad
 \widehat V_{T_+}\simeq\widehat W_T.
\]
Completing this last isomorphism further along $\Gamma_i\subset T_+$, and using transitivity of completion, gives
\begin{equation}\label{eq:formal-neighborhood-Ci}
 \widehat V_{\Gamma_i}
 \simeq
 \widehat{\Tot_{\Pp^1}
 (\Oo_{\Pp^1}(-1)\oplus\Oo_{\Pp^1}(-d_i))}_{\Pp^1}.
\end{equation}

By \eqref{eq:exceptional-locus-fi}, \eqref{eq:formal-neighborhood-Ci}, and Lemma~\ref{lem:embedding-dimension}, the point $y_i$ is the unique singular point of $Y_i$, and
\[
 \edim_{y_i}Y_i=d_i+3.
\]
If $Y_i\simeq Y_j$, then an isomorphism must send the unique singular point $y_i$ to $y_j$. Since embedding dimension is invariant under isomorphism, this would imply $d_i=d_j$. The integers $d_i$ are strictly increasing by \eqref{eq:di-increasing}; hence the varieties $Y_i$ are pairwise non-isomorphic. Therefore $\CC^3(V)$ is infinite.
\end{proof}

\begin{remark}
In the above construction, the double cover makes the resulting threefold canonically
polarized. It is not needed to produce infinitely many log canonical
models: the same construction can be carried out directly
on \(W\).
\end{remark}

\begin{remark}\label{rmk:Q-fact and terminal not enough}
The counterexample in Theorem~\ref{thm:intro-threefold} can be strengthened by taking the flips of the small contractions $f_i\colon V\to Y_i$. Since formally \(N_{\Gamma_i/V}\simeq \mathcal O_{\mathbb P^1}(-1)\oplus\mathcal O_{\mathbb P^1}(-d_i)\), the flipped variety \(V_i^+\) has a unique cyclic quotient singularity of type
\[
\frac1{d_i}(1,1,d_i-1),
\]
which is terminal and has canonical index \(d_i\). Hence the \(V_i^+\) are \(\mathbb Q\)-factorial terminal and pairwise non-isomorphic. After adding a sufficiently ample pullback to the boundary, each \(V_i^+\) is realized as a log canonical model of a klt pair on \(V\). Thus, the assumption of \(\mathbb Q\)-factorial terminal singularities is insufficient to obtain finiteness.
\end{remark}
 
\bibliographystyle{alpha}
\bibliography{bibfile}

\end{document}